\documentclass[10pt,oneside]{article}    
  
\usepackage{mathrsfs}
\usepackage{amsmath,amsfonts,amssymb,amsthm}
\usepackage{caption}
\usepackage{subfigure}
\usepackage{cite}
\usepackage{color}
\usepackage{graphicx}
\usepackage{subfigure}
\usepackage{float}
\usepackage{paralist}
\usepackage[colorlinks,
linkcolor=red,
citecolor=blue
]{hyperref}
\usepackage{authblk}
\usepackage[T1]{fontenc}

\newtheorem{thm}{Theorem}[section]
\newtheorem{lem}[thm]{Lemma}
\newtheorem{prop}[thm]{Proposition}
\newtheorem{rmk}{Remark}[section]
\def\f{\frac}
\def\l{\lambda}
\def\a{\alpha}
\def\gm{\gamma}
\def\r{\rho }
\def\p{\partial}
\def\b{\bar}

\def\e{\eta}
\def\h{\phi}
\def\s{\psi}
\def\b{\bar}
\def\v{\varepsilon}
\def\wt{\widetilde}
\def\be{\begin{equation}}
\def\ee{\end{equation}}
\def\bs{\begin{split}}
\def\es{\end{split}}	
\def\mb{\mathbb}
\def\mc{\mathcal}
\def\bma#1\ema{{\allowdisplaybreaks\begin{split}#1\end{split}}}

\numberwithin{equation}{section}

\begin{document}
\title{Existence and Nonlinear Stability of  Stationary Solutions to the Full Two-Phase Flow Model in a Half Line }
\author[a,b]{Hai-Liang  Li  \thanks{
E-mail:		hailiang.li.math@gmail.com (H.-L. Li)}} 

\author[a,b]{Shuang Zhao \thanks{E-mail: shuangzhaomath@163.com(S. Zhao)}}
    \affil[a]{School of Mathematical Sciences,
	Capital Normal University, Beijing 100048, P.R. China}
\affil[b]{Academy for Multidisciplinary Studies, Capital Normal University, Beijing 100048, P.R. China}
\date{}
\renewcommand*{\Affilfont}{\small\it}	
\maketitle
\begin{abstract}The inflow  problem for the full two-phase model in a half line is investigated in this paper. 
 The existence and uniqueness of the stationary solution is shown by  applications of center manifold theory, and its nonlinear stablility of the stationary solution is established for the small perturbation.
\end{abstract}
\section{Introduction} 
Two-phase flow models play  important roles in applied scientific areas, for instance,  nuclear, engines, chemical engineering, medicine, oil-gas,  fluidization,   waste water treatment, liquid crystals,   lubrication, biomedical flows \cite{IM,MV,HL}, etc. 
In this paper, we consider the full two-phase flow model  which  
can be  formally obtained from a Vlasov-Fokker-Planck equation coupled with the compressible Navier-Stokes equations through the  Chapman-Enskog expansion \cite{LWW}.
  
We consider the initial-boundary value problem (IBVP) for the  two-fluid model as follows:
\begin{equation}
\left\{
\begin{aligned}
&                  
\rho_{t}+(\rho u)_{x}=0,
\\&
(\rho u)_{t} +[\rho u^{2} + p_{1}(\rho)  ]_{x}      =(\mu  u_{x})_{x}  +n(v-u)
\\&
n_{t}+(n v)_{x}=0,
\\&
(n v)_{t} +[n v^{2} + p_{2}(n)  ]_{x}      =(n v_{x})_{x}  -n(v-u)
, ~~~x>0,~t>0,
\label{f}  
\end{aligned}   
\right . 
\end{equation}  
where $\rho>0$ and $n>0$ stand for the densities, $u$ and $v$ are the velocities of two fluids, and the constant $\mu$ is the viscosity coefficient. The pressure-densities functions take forms
\be 
p_{1}(\rho)=A_{1}\rho^{\gamma}, \quad p_{2}(n)=A_{2}n^{\alpha}\ee
 with  four constants $A_{1}>0 ,~ A_{2}>0 $, $ \gamma \geq  1$ and $\alpha\geq 1$.   
 The initial data satisfy
 \begin{equation}
 ( \rho,  u, n, v)(0,  x)=  ( \rho_{0}, u_{0}, n_{0}, v_{0})(x),      ~~~~\underset{x\in \mathbb{R}_{+}}{\inf}   \rho _{0}(x) > 0,    ~~~~\underset{x\in \mathbb{R}_{+}}{\inf}   n _{0}(x)  > 0 ,
 \label{initial d}
 \end{equation}
 \begin{equation}
 \begin{split}
 &
 \lim_{x\to+\infty} ( \rho_{0}, u_{0}, n_{0}, v_{0})(x)        =(\rho_{+},u_{+}, n_{+}, u_{+}), ~~~~~\rho_{+}> 0,~~~~~n _{+}> 0 , ~~~~~~~~      
 \label{initial c}
 \end{split}
 \end{equation}
  and inflow boundary condition imposed with
  \begin{equation}
  (\r,u,n,v)(t,0)=(\r_{-},u_{-},n_{-},u_{-})
  \label{inflow c}
  \end{equation}
  where $\rho_{+}>0$, $n_{+}>0$ and $u_{-}>0$ are constants.
  
  
 The main purpose of this paper is to prove the  existence and nonlinear  stability of the stationary solution in Sobolev space. The stationary solution
$(\widetilde{\rho },\widetilde{u},\wt{n},\wt{v})(x)$ corresponding to the problem $(\ref{f})$-$(\ref{inflow c})$ 
satisfies the following system 
\begin{equation}
\left\lbrace 
\begin{aligned}
&
(  \widetilde{\rho  }   \widetilde{u}   )_{x}  =0,
\\
&
[\widetilde{\rho}  \widetilde{u}^{2}+p_{1}(\widetilde{\rho})]_{x}
=(\mu \widetilde{u}_{x})_{x}   +\widetilde{n}( \widetilde{v}- \widetilde{u}),
\\
&
(  \widetilde{n  }   \widetilde{v}   )_{x}  =0,
\\
&
[\widetilde{n}  \widetilde{v}^{2}+p_{2}(\widetilde{n})]_{x}
=(\widetilde{n}\widetilde{v}_{x})_{x}   -\widetilde{n}( \widetilde{v}- \widetilde{u}),
 ~~~x>0,~
\end{aligned}
\right. 
\label{stationary f}
\end{equation}
and the boundary condition and spatial far field condition
\begin{equation}
\begin{split} 
&
(\wt{\r},\wt{u},\wt{n},\wt{v})(0)=(\r_{-},u_{-},n_{-},v_{-}),
~~~~\lim_{x \to \infty}
(\widetilde{\rho },\widetilde{u},\widetilde{n },\widetilde{v })(x)
=(\rho_{+},u_{+},n_{+},u_{+}),
\end{split} 
\end{equation}
\begin{equation}
\begin{split} 
~~~~\underset{x\in \mathbb{R}_{+}}  {\inf }\widetilde{n}(x)   >0,
~~~~\underset{x\in \mathbb{R}_{+}}  {\inf}  \widetilde{\rho}(x)>0.
\label{stationary boundary c}
\end{split} 
\end{equation}
 Integrating
 $ (\ref{stationary f})_{1}$,
 $ (\ref{stationary f})_{3}$ 
 over
  $(x, +\infty)  $ and $(0,x)$, 
 we obtain
 \begin{equation}
 \begin{split}
 &
\wt{u}=\f{\r_{+}u_{+}}{\wt{\r}}=\f{\r_{-}u_{-}}{\wt{\r}},\quad \wt{v}=\f{n_{+}u_{+}}{\wt{n}}=\f{n_{-}u_{-}}{\wt{n}}
\label{wt r u}
\end{split} 
 \end{equation}
 which implies that the relationships
 \begin{equation}
 u_{+} =\frac{\r_{-}}{\rho_{+}} u_{-} =\f{n_{-}}{n_{+}} u_{-}>0, \quad
  \f{u_{+}}{u_{-}} =\f{\r_{-}}{\r_{+}} =\f{n_{-}}{n_{+}}
 \label{u_{+}}
 \end{equation}
 are  necessary for the existence of  stationary solutions to the boundary value  problem (BVP)
 $ (\ref{stationary f})$-$ (\ref{stationary boundary c}) $.\\
 Define the Mach number $M_{+}$ and sound speed $c_{+}$ at the spatial far field as follows
\be M_{+}:=\f{|u_{+}|}{c_{+}},\quad c_{+}:=(\f{A_{1}\gm \r_{+}^{\gm}+A_{2}\a n_{+}^{\a}}{\r_{+}+n_{+}})^{\f{1}{2}}. \ee
$~~~~$Then, we have the  following results about the existence and uniqueness  of the stationary solution.
\begin{thm}
	\label{thm stationary s}
	Assume that   $\delta:=|u_{+}-u_{-}|>0$,  $ u_{+}>0,
	\f{u_{+}}{u_{-}}= \f{\r_{-}}{\r_{+}}=\f{n_{-}}{n_{+}}$ hold.  Then there exists  a set $\Omega_{+}\subset \mb{R}_{+}$  such that if $u_{-} \in \Omega_{+}$ and 
	$\delta$ sufficiently small,   there exists a unique strong solution $(\widetilde{\rho },\widetilde{u},\widetilde{n},\widetilde{v})$ to the 
	problem $(\ref{stationary f})$-$(\ref{stationary boundary c})$  which satisfies either for the supersonic or subsonic case $M_{+}\neq 1$ that
	\begin{equation}
	~~~~~~|   \partial_{x}^{k} (\widetilde{\rho}-\rho_{+},\widetilde{u}-u_{+},\widetilde{n}-n_{+},\widetilde{v}-u_{+})  |   \leq C \delta e^{-m x} ,~~
	~k=0,1,2,
	\label{M_{+}>1 stationary solution d}
	\end{equation}
	or for the sonic case $M_{+}=1$ that    $\wt{u}_{x}\geq 0$, $\wt{v}_{x}\geq 0$ and
	\begin{equation}
	|\partial_{x}^{k}  (\widetilde{\rho}-\rho_{+},\widetilde{u}-u_{+},\widetilde{n}-n_{+},\widetilde{v}-u_{+}) |  \leq  C \frac{\delta ^{k+1}}{(1+\delta x)^{k+1}},~k=0,1,2, \label{sigma}\end{equation}
	where $C>0$, $m>0$ are positive constants.  
\end{thm} 
\begin{rmk}
	In contrast to isentropic Navier-Stokes equations {\rm \cite{MN in,HMS in}} and full compressible Navier-Stokes equations{\rm \cite{WQ in,NN in}}, where there is no stationary solutions for the supersonic case, there exist stationary solutions for supersonic, subsonic and sonic cases to the IBVP $(\ref{f})$-$(\ref{inflow c})$. 
\end{rmk}
Then, we   have   the  nonlinear stability of the stationary solution for $(\ref{f})$-$(\ref{inflow c})$ as follows. 
\begin{thm}\label{thm long time behavior}Let  the same conditions in Theorem $\ref{thm stationary s}$ hold and assume that it holds 
	$$|p{'}_{1}({\r_{+}})-p{'}_{2}({n_{+}})|	\leq  \sqrt{2}u_{+}\min\{\ (1+\f{\r_{+}}{n_{+}})[(\gm-1)p{'}_{1}(\r_{+})]^{\f{1}{2}},(1+\f{n_{+}}{\r_{+}})[(\a-1)p{'}_{2}(n_{+})]^{\f{1}{2}} \}$$
	 for the sonic case $M_{+}=1$.
	Then, there exists a positive constant $\varepsilon_{0}$ such that 
	\be \| (\rho_{0}-\widetilde{\rho}, u_{0}-\widetilde{u}, n_{0}-\widetilde{n},v_{0}-\widetilde{ v}) \|_{H^{1}}+\delta\leq \varepsilon_{0},
	\nonumber\ee
	the  problem $(\ref{f})$-$(\ref{inflow c})$ has a unique global solution $(\rho,u, n, v)(t,x)$ satisfying
	  \begin{equation}
	\left\lbrace 
	\begin{split}&
	(\rho-\widetilde{\rho},u-\widetilde{u},n-\widetilde{n},v-\widetilde{v}   )\in C([0,+\infty);H^{1}),\\&
	(\rho-\widetilde{\rho},n-\widetilde{n} )_{x}\in L^{2}([0,+\infty);L^{2}) ,\\&
	(u-\widetilde{u},v-\widetilde{v} )_{x} \in L^{2}([0,+\infty);H^{1}),
	\nonumber
	\end{split}
	\right. 
	\end{equation}
	and 
	\begin{equation}
	\lim_{t \to +\infty } \sup_{x\in \mathbb{R_{+}}}|(\rho,u,n,v)(t,x)-(\widetilde{\rho},\widetilde{u},\widetilde{n},\widetilde{v})(x)|=0.
	\end{equation}
\end{thm}


\vspace{2ex}
\noindent{\textbf{Notation.}}      
$~~$
We denote by $\| \cdot\|_{L^{p}}$ the norm of the usual Lebesgue space $L^{p}=L^{p}({\mathbb{R}_{+}})$,  $1 \leq p \leq \infty$. And if $p=2$, we write $\| \cdot \|_{L^{p}(\mathbb{R}_{+})}=\|\cdot\|$.
 $H^{s}(\mathbb{R}_{+})$ stands for the standard $s$-th  Sobolev space over $\mathbb{R}_{+}$ equipped with its norm
 $\| f \|_{H^{s}(\mathbb{R}_{+})}=\| f \|_{s}=( \sum\limits_{i=0}^{s} \| \partial^{i}f \|^{2})^{\frac{1}{2}}$. $C([0,T]; H^{1}(\mathbb{R}_{+}))$ represents the space of continuous functions on the interval $[0,T]$ with values in $H^{s}(\mathbb{R}_{+})$. 

\vspace{2ex}

The rest of this paper will be organized as follows.
 We investigate the existence and uniquenesss of  the stationary solution in Section \ref{sec:2},
and gain the nonlinear stability of the solution in Section \ref{sec:$M_{+} >1$}.
\section{Existence of Stationary Solution}
\label{sec:2}
We prove Theorem \ref{thm stationary s} on the the existence and uniqueness of the stationary solution to $(\ref{stationary f})$-$(\ref{stationary boundary c})$  with	$u_{+}>0$ and $\delta$ sufficiently small as follows. The natural idea is to apply the center manifold theory\cite{C} to the BVP $({\ref{stationary f}})$-$(\ref{stationary boundary c})$, 
where  it is necessary to  get the bound estimates of $\wt{u}_{x}(0)$ or $\wt{v}_{x}(0)$. 
\begin{lem}
	Assume that   $ u_{+}>0,
	\f{u_{+}}{u_{-}}= \f{\r_{-}}{\r_{+}}=\f{n_{-}}{n_{+}}$,   $\delta:=|u_{+}-u_{-}|>0$ hold with $\delta$ sufficiently small. 
Then the solution $(\wt{\r},\wt{u},\wt{n},\wt{v})$  to the BVP $({\ref{stationary f}})$-$(\ref{stationary boundary c})$ satisfies
\be|\wt{u}_{x}(0)|\leq C|u_{-}-u_{+}|, \quad|\wt{v}_{x}(0)|\leq C|u_{-}-u_{+}|,\ee where $C>0$ is 	a positive constant.
\end{lem}
\begin{proof}
Due to $\wt{\r}=\f{\r_{+}u_{+}}{\wt{u}}$ and $\wt{n}=\f{n_{+}u_{+}}{\wt{v}}$,  we obtain
\be
\left\lbrace 
\begin{split}
&(\r_{+}u_{+}\wt{u}+A_{1}\r^{\gm}_+ u^{\gm}_+ \wt{u}^{-\gm})_{x}=(\mu \wt{u}_{x})_{x}  +\f{n_+ u_{+}}{\wt{v}}(\wt{v}-\wt{u}),\\&
(n_{+}u_{+}\wt{v}+A_{2}n^{\a}_+  u^{\a}_+ \wt{v}^{-\a})_{x}=(n_+ u_+ \f{\wt{v}_{x}}{\wt{v}})_x -\f{n_+ u_{+}}{\wt{v}}(\wt{v}-\wt{u}).
\end{split}
\right. 
\label{sf}
\ee
Adding  $(\ref{sf})_{1}$ to $(\ref{sf})_{2}$ together,  then integrating the resulted equation over $(0,\infty)$, we have 
\be \mu \wt{u}_{x}(0)+\f{n_{+}u_{+}}{u_{-}}\wt{v}_{x}(0)=\f{1}{u_+}[(\r_{+}+n_{+}) u^{2}_+ - (A_{1}\gm \r^{\gm}_+ +A_{2} \a n^{\a}_+)] (u_{-} -u_{+})+O(|u_{-}-u_{+}|^{2}) .
\label{wt u v}\ee 
Multiplying  $(\ref{sf})_{1}$ by $\wt{u}$, $(\ref{sf})_{2}$ by $\wt{v}$ respectively,  then adding them together and intergrating the resulted equation over $(0, \infty)$ lead to 
\be 
\begin{aligned}
&
\int_{0}^{\infty} (\mu \wt{u}_{x}^{2}+n_+ u_+\f{\wt{v}^{2}_{x}}{\wt{v}})dx
+\int_{0}^{\infty}\f{n_+ u_+}{\wt{v}}(\wt{v}-\wt{u})^{2}dx
\\
=&
-u_{-}(\mu \wt{u}_{x}(0)+n_{+}u_{+}\f{\wt{v}_{x}(0)}{\wt{v}})+
[(\r_{+}+n_{+})u_{+}^{2}-(A_{1}\gm \r^{\gm}_{+}+A_{2}\a n^{\a}_{+} )](u_{-}-u_{+})+O(|u_{-}-u_{+}|^{2})
\\
=&O(|u_{-}-u_{+}|^{2}),
\end{aligned}
\ee 
where we have used $(\ref{wt u v})$.\\
Multiplying $(\ref{sf})_{1}$ by $\wt{u}^{b} (b<-\max\{\a,\gamma\}-1)$, $(\ref{sf})_{2}$ by $\wt{v}^{b}$ respectively,  then adding them togther and  intergrating the resulted equation over $(0, \infty)$ lead to 
\begin{align}
&
\int_{0}^{\infty} (-b\mu \wt{u}^{2}_{x}\wt{u}^{b-1}-bn_+ u_+\f{\wt{v}^{2}_{x}}{\wt{v}}\wt{v}^{b-1})dx
+\int_{0}^{\infty}\f{n_+ u_+}{\wt{v}}(\wt{v}-\wt{u})(\wt{u}^{b}-\wt{v}^{b})dx
\notag
\\
=&
u^{b}_{-}(\mu \wt{u}_{x}(0)+n_{+}u_{+}\f{\wt{v}_{x}(0)}{\wt{v}})-u_{-}^{b-1}[(\r_{+}+n_{+})u_{+}^{2}-(A_{1}\gm \r^{\gm}_{+}+A_{2}\a n^{\a}_{+} )](u_{-}-u_{+})
+O(|u_{-}-u_{+}|^{2})
\notag
\\
=&
O(|u_{-}-u_{+}|^{2}),
\end{align}
where we have used $(\ref{wt u v})$.\\
Multiplying $(\ref{sf})_{2}$ by $\f{\wt{v}_{x}}{\wt{v}}$ and then intergrating the resulted equality over $(0, \infty)$ yield 
\be
\begin{aligned}
&
n_{+} u_{+}\f{\wt{v}^{2}_{x}(0)}{2u^{2}_{-}}
+
\int_{0}^{\infty}n_{+} u_{+}\f{\wt{v}^{2}_{x}}{\wt{v}}dx
\\
=&
\int_{0}^{\infty}A_{2}\a n^{\a}_{+}\f{1}{u^{2}_{+}}(\f{\wt{v}}{u_{+}})^{-(\a +2)}\wt{v}^{2}_{x}dx -\int_{0}^{\infty}\f{n_+ u_+}{\wt{v}}(\wt{v}-\wt{u})\f{\wt{v}_{x}}{\wt{v}}dx
:=\sum_{i=1}^{2}\mathcal{I}_{i}.
\label{sf'b3}
\end{aligned}
\ee
We turn to estimate terms in the right hand side of $(\ref{sf'b3})$. Using  
$(\ref{wt u v})$, we have 
\be 
\begin{split}
|I_{1}|+|I_{2}|\leq C\int_{0}^{\infty}\f{\r_+ u_+}{\wt{v}}(\wt{v}-\wt{u})^{2}dx+C\int_{0}^{\infty}\wt{v}^{b-2}\wt{v}^{2}_{x}dx\leq C|u_{-}-u_{+}|^{2},
\label{sf''b2}
\end{split}
\ee 
With the help of  $(\ref{sf'b3})$,  $(\ref{sf''b2})$ and $(\ref{wt u v})$, we get 
\be
|\wt{v}_{x}(0)|\leq C|u_{-}-u_{+}|, \quad |\wt{u}_{x}(0)|\leq C|u_{-}-u_{+}| .
\label{sf1}
\ee
\end{proof}
\vspace{2ex}
We use the  manifold theory \cite{C}
to obtain the existence and uniqueness of the stationary solution $(\ref{stationary f})$-$(\ref{stationary boundary c})$ of full two-phase flow  model. 
Firstly, it is necessary to reformulate the system $(\ref{stationary f})$ into a $3\times 3$  system of ordinary differential equations of the  first-order. 

Adding $(\ref{stationary f})_{2}$, $(\ref{stationary f})_{4}$ together and  substituting $\wt{\r}=\f{\r_{+}u_{+}}{\wt{u}}$, $\wt{n}=\f{n_{+}u_{+}}{\wt{v}}$ into the resulted equation, then integrating the resulted equality over $(x,\infty)$, we gain 
\be 
\r_{+}u_{+}(\wt{u}-1)+A_{1}\r^{\gm}_{+}u^{\gm}_{+}(\wt{u}^{-\gm}-1)+n_{+}u_{+}(\wt{v}-1)+A_{2}n^{\a}_{+}u^{\a}_{+}(\wt{v}^{-\a}-1)
=\mu \wt{u}_{x}+n_{+}u_{+}\f{\wt{v}_{x}}{\wt{v}}.
\ee
 We consider the following system
\begin{equation}
\left\lbrace 
\begin{aligned}
&
(\r_{+}u_{+}\wt{u} +A_{1}\r^{\gm}_{+}u^{\gm}_{+}\wt{u}^{-\gm})_{x}
=(\mu \wt{u}_{x})_{x} +n_{+}u_{+}(1 -\f{\wt{u}}{\wt{v}}),
\\
&
\r_{+}u_{+}(\wt{u} -1) +A_{1}\r^{\gm}_{+}u^{\gm}_{+}(\wt{u}^{-\gm} -1) +n_{+}u_{+}(\wt{v} -1) +A_{2}n^{\a}_{+}u^{\a}_{+}(\wt{v}^{-\a} -1)
=\mu \wt{u}_{x} +n_{+}u_{+}\f{\wt{v}_{x}}{\wt{v}}.
\end{aligned}
\label{2sf}
\right. 
\end{equation}
It is easy to see that the system $(\ref{2sf})$ contains a second-order equation and a first-order equation.
 Hence, the system $(\ref{2sf})$ can be  reformulated into the following  system 
\be 
\left\lbrace 
\begin{aligned}
&
\wt{u}_{x}=\wt{w},
\\
&
\wt{w}_{x}= \f{1}{\mu}[\r_{+}u_{+}\wt{w} -A_{1}\r^{\gm}_{+}u^{\gm}_{+}\wt{u}^{-\gm-1}\wt{w} -n_{+}u_{+}(1-\f{\wt{u}}{\wt{v}})],
\\
&
\wt{v}_{x} =\f{\wt{v}}{n_{+}u_{+}}[ \r_{+}u_{+}(\wt{u} -u_{+}) +A_{1}\r^{\gm}_{+}(\f{u^{\gm}_{+}}{\wt{u}^{\gm}} -1) +n_{+}u_{+}(\wt{v} -u_{+}) +A_{2}n^{\a}_{+}u^{\a}_{+}(\f{u^{\a}_{+}}{\wt{v}^{\a}} -1) -\mu \wt{w}].
\end{aligned}
\right. 
\label{3sf}
\ee 
The boundary condition satisfies 
\be
(\wt{u},\wt{w},\wt{v})(0)=(u_{-},\wt{u}_{x}(0),u_{-}),\quad \lim_{x\to \infty}(\wt{u},\wt{w},\wt{v})(x)=(u_{+},0,u_{+}).
\label{sf c}
\ee 
Define the perturbation near the far filed state $(u_{+},0,u_{+})$ as   $(\bar{u},\bar{w},\bar{v})=(\wt{u}-u_{+},\wt{w},\wt{v}-u_{+})$. The system $(\ref{3sf})$, $(\ref{sf c})$ can be rewritten as follows:
\begin{equation}
\left\lbrace 
\begin{split}
&\frac{d}{dx}
\begin{pmatrix}
\bar{u}\\ 
\bar{w}\\ 
\bar{v}
\end{pmatrix}
=\boldsymbol{J_{+}}\begin{pmatrix}
\bar{u}\\ 
\bar{w}\\ 
\bar{v}
\end{pmatrix}
+\begin{pmatrix}
\bar{g}_{1}(\bar{u},\bar{w},\bar{v})\\ 
\bar{g}_{2}(\bar{u},\bar{w},\bar{v})\\ 
\bar{g}_{3}(\bar{u},\bar{w},\bar{v})
\end{pmatrix},\\&
(\bar{u},\bar{w}, \bar{v})(0):=(\b{u}_{-},\bar{w}_{-},\bar{v}_{-})=(u_{-}-u_{+},\wt{u}_{x}(0),u_{-}-u_{+}), \quad \lim_{x\to\infty} (\bar{u},\bar{w}, \bar{v})=(0,0,0),
\end{split}
\right. 
\label{sf bar}
\end{equation}
where the  matrix $\boldsymbol{J_{+}}$ is the   defined as follows:
\be
\boldsymbol{J_{+}}=
\begin{pmatrix}
	0 &1  &0 \\ 
	\f{n_{+}}{\mu }& \frac{1}{\mu u_{+}}(\r_{+}u_{+}^{2}-A_{1}\gamma \rho^{\gamma}_{+}) &-\f{n_{+}}{\mu} \\ 
	\frac{1}{n_{+}u_{+}}(\r_{+}u_{+}^{2}-A_{1}\gamma \rho^{\gamma}_{+})  &-\f{\mu}{n_{+}}  & \frac{1}{n_{+}u_{+}}(n_{+}u_{+}^{2}-A_{2}\alpha n^{\alpha}_{+}) 
\end{pmatrix}
\ee
and  $\bar{g}_{1},\bar{g}_{2}, \bar{g}_{3}$ are nonlinear functions defined by
\be
\begin{aligned}
\bar{g}_{1}(\bar{u},\bar{w},\bar{v})
=&0,
\\
\bar{g}_{2}(\bar{u},\bar{w},\bar{v})=&
\f{1}{2}(2\f{n_{+}}{\mu}\f{1}{u_{+}}\bar{v}^{2}-2\f{n_{+}}{\mu}\f{1}{u_{+}}\bar{u}\bar{v}+2\f{A_{1}\gm (\gm +1)\r^{\gm}_{+}}{\mu u^{2}_{+}}\bar{w}\bar{u})+O(|\bar{v}|^{3}+|\bar{u}|^{3}|+|\bar{w}|^{3}),
\\
\bar{g}_{3}(\bar{u},\bar{w},\bar{v})=&
\f{1}{2}[ \f{A_{1}\gm(\gm+1)\r^{\gm}_{+}}{n_{+}u^{2}_{+}}\bar{u}^{2}
+2\f{\r_{+}u_{+}^{2}-A_{1}\gm\r^{\gm}_{+}}{n_{+}u^{2}_{+}}\bar{v} \bar{u} +(2\f{n_{+}u_{+}^{2}-A_{2}\a n^{\a}_{+}}{n_{+}u^{2}_{+}}+\f{A_{2}\a(\a+1)n^{\a}_{+}}{n_{+}u^{2}_{+}})\bar{v}^{2}
\\
&
-2\f{\mu}{n_{+}u_{+}}\bar{w}\bar{v}]+O(|\bar{v}|^{3}+|\bar{u}|^{3}|+|\bar{w}|^{3}).
\end{aligned}
\label{g}
\ee
 Three eigenvalues $\l_{1}, \l_{2}, \l_{3}$ of matrix $\boldsymbol{J_{+}}$ satisfy 
\be
\left\lbrace 
\begin{aligned}
\l&_{1}\l_{2}\l_{3}=-\f{(n_{+}+\r_{+})u^{2}_{+}-(A_{1}\gm \r^{\gm}_{+}+A_{2}\a \r^{\a}_{+} )}{\mu u_{+}},
\\
&
 \l_{1}+\l_{2}+\l_{3}=\f{\r_{+}u^{2}_{+}-A_{1}\gm \r^{\gm}_{+}}{\mu u_{+}}+\f{n_{+}u^{2}_{+}-A_{2}\a \r^{\a}_{+}}{n_{+} u_{+}},
 \\
 &
\l_{1}\l_{2}+\l_{1}\l_{3}+\l_{2}\l_{3}=\r_{+}\f{(u_{+}^{2}-A_{1}\gm\r_{+}^{\gm-1})(u_{+}^{2} -A_{2}\a\r_{+}^{\a-1})}{\mu u^{2}_{+}}-1-\f{n_{+}}{\mu}.
\end{aligned}
\right. 
\label{l f1}
\ee 
If $M_{+}>1$, it is easy to check
\be
\l_{1}\l_{2}\l_{3}<0, 
\label{l3}
\ee
and
\be
u_{+}^{2}>\min\{A_{1}\gm\r_{+}^{\gm-1}, A_{2}\a n_{+}^{\a-1}\}.
\ee
Moreover,
  we obtain
\be
\l_{1}+\l_{2}+\l_{3}>0, \quad {\rm for}\quad u^{2}_{+}>\max\{ A_{1}\gm\r_{+}^{\gm-1}, A_{2}\a n_{+}^{\a-1}\}
\label{l1}
\ee
and   get
\be 
\l_{1}\l_{2}+\l_{1}\l_{3}+\l_{2}\l_{3}<0,\quad {\rm for}\quad  u^{2}_{+}>\max\{ A_{1}\gm\r_{+}^{\gm-1}, A_{2}\a n_{+}^{\a-1}\}.
\label{l2}
\ee
Due to $(\ref{l3})$, $(\ref{l1})$, $(\ref{l2})$, we  have
\be 
{\rm Re}\l_{1}>0, {\rm Re}\l_{2}>0, \l_{3}<0, \quad {\rm for }~M_{+}>1.
\label{M+}
\ee 
Using the similar arguments,  we have the following results:
\be 
\left\lbrace 
\begin{aligned}
	&{\rm if}~ M_{+}>1,{ \rm then}~ {\rm Re}\l_{1}>0, {\rm Re}\l_{2}>0, \l_{3}<0,\\&
	{\rm if}~ M_{+}<1, {\rm then}~{\rm Re}\l_{1}<0, {\rm Re}\l_{2}<0, \l_{3}>0, \\&
	{\rm if} ~M_{+}=1, {\rm then} ~\l_{1}>0, \l_{2}<0, \l_{3}=0. 
\end{aligned}
\right. 
\label{M_+}
\ee
In order to prove the existence of the solution $(\bar{u},\bar{w},\bar{v})$ to  the BVP $(\ref{sf bar})$, we need to diagonalize the system $(\ref{sf bar})$.
Take  a linear  coordinate transformation
\be 
\begin{pmatrix}
	z_{1}\\ 
	z_{2}\\ 
	z_{3}
\end{pmatrix}
=\boldsymbol{P}^{-1}\begin{pmatrix}
	\bar{u}\\ 
	\bar{w}\\ 
	\bar{v}
\end{pmatrix},
\label{trans}
\ee
where $()^{T}$ denotes the transpose of a row vector and the    invertible matrix $\boldsymbol{P}$ satisfies
\be 
\boldsymbol{P}^{-1}\boldsymbol{J_{+}}\boldsymbol{P}=
\begin{pmatrix}
	\lambda_{1} &*  &0 \\ 
	0 & \lambda_{2} &0 \\ 
	0 &0  & \lambda_{3}
\end{pmatrix}.
\ee 
Thus, we have 
 \be 
 \begin{aligned}
&\frac{d}{dx}\begin{pmatrix}
	z_{1}\\ 
	z_{2}\\ 
	z_{3}
\end{pmatrix}=\begin{pmatrix}
	\lambda_{1} & * &0 \\ 
	0& \lambda_{2} &0 \\ 
	0& 0 & \lambda_{3}
\end{pmatrix}
\begin{pmatrix}
	z_{1}\\ 
	z_{2}\\ 
	z_{3}
\end{pmatrix}
+
\begin{pmatrix}
	g_{1}(z_{1},z_{2},z_{3})\\ 
	g_{2}(z_{1},z_{2},z_{3}) \\ 
	g_{3}(z_{1},z_{2},z_{3})
\end{pmatrix},\\&
(z_{1},z_{2},z_{3})(0)=(z_{1-},z_{2-},z_{3-}),\quad \lim_{x\to \infty}(z_{1},z_{2},z_{3})=(0,0,0),
\label{sf z}
 \end{aligned}
 \ee
 where nonlinear functions $g_{i}(i=1,2,3)$ and the boundary condition $(z_{1-},z_{2-},z_{3-})$ are denoted by
 \be
 \begin{pmatrix}
 	g_{1}(z_{1},z_{2},z_{3})\\ 
 	g_{2}(z_{1},z_{2},z_{3}) \\ 
 	g_{3}(z_{1},z_{2},z_{3})
 \end{pmatrix}
=\boldsymbol{P}^{-1}
\begin{pmatrix}
	0\\ 
	\bar{g}_{2}(\bar{u},\bar{w},\bar{v})\\ 
	\bar{g}_{3}(\bar{u},\bar{w},\bar{v})
\end{pmatrix},
\quad 
\begin{pmatrix}
	z_{1-}\\ 
	z_{2-} \\ 
	z_{3-}
\end{pmatrix}=
\boldsymbol{P}^{-1}
\begin{pmatrix}
	\b{u}_{-}\\ 
	\bar{w}_{-} \\ 
	\bar{v}_{-}
\end{pmatrix}
\label{bg}
 \ee 
 Then, we show case (i) and case (ii) in Theorem \ref{thm stationary s}.
 
${\rm (i)}$ 
For the supersonic case $M_{+}>1$  which satisfies ${\rm Re}\l_{1}>0,{\rm Re}\l_{2}>0,\l_{1}>0$ and  ${\rm Re}\l_{3}<0$, there exists a one-dimension local stable manifold
\be W^{s}(0,0,0)=\{(z_{1},z_{2},z_{3})~ | z_{1}=h^{s}_{1}(z_{3}), z_{2}=h^{s}_{2}(z_{3}), |z_{3}|~{\rm sufficiently~small}   \},\ee
where $h^{s}_{i}, i=1,2$ are smooth functions and $h^{s}_{i}(0)=0,~ Dh^{s}_{i}(0)=0,~i=1,2$.
Therefore, according to manifold theorem, if  
$(z_{1-},z_{2-},z_{3-})\in W^{s}(0,0,0)$, then there exists a unique solution to problem $(\ref{sf bar})$ satisfying 
\be 
|\p^{k}(z_{1},z_{2},z_{3})|\leq C\delta e^{-mx}
\label{ed}
\ee
where we used $(z_{1-},z_{2-},z_{3-})^{T}=\boldsymbol{P}^{-1}(\bar{u}_{-},\b{w}_{-},\b{v}_{-})^{T}$ and $|(\b{u}_{-},\b{w}_{-},\bar{v})|\leq C\delta$.

${\rm (ii)}$  
For  the subsonic case $M_{+}<1$  which leads to  $\l_{1}>0,{\rm Re}\l_{2}<0, {\rm Re}\l_{3}<0$,  there exists a two-dimension local stable manifold
\be W^{s}_{2}(0,0,0)=\{(z_{1},z_{2},z_{3})~ | z_{3}=g^{s}(z_{1},z_{2}),  |(z_{1},z_{2})|~{\rm sufficiently ~small}   \},\ee
where $g^{s}$ is a smooth function and $g^{s}(0)=0,~ Dg^{s}(0)=0$. Thus, according to manifold theorem, if  
$(z_{1-},z_{2-},z_{3-})\in W_{2}^{s}$, then there exists a unique solution to problem $(\ref{sf bar})$ satisfying $(\ref{ed})$.

Finally, we prove  case (iii) in Theorem \ref{thm stationary s}.

${\rm (iii)}$ 
We consider the sonic case $M_{+}=1$  which implies $\l_{1}>0, \l_{2}<0, \l_{3}=0$. Moreover, we have  
\be
\left\lbrace 
\begin{split}
&\l_{1}+\l_{2}=\f{\r_{+}}{\mu}\f{u^{2}_{+}-A_{1}\gm\r^{\gm-1}}{u_{+}}+\f{u^{2}_{+}-A_{2}\a n^{\a-1}}{u_{+}},\\&
\l_{1}\l_{2}=\f{\r_{+}}{\mu u^{2}_{+}}(u^{2}_{+}-A_{1}\gm \r^{\gm-1}_{+})(u^{2}_{+}-A_{2}\a n^{\a-1}_{+})-(1+\f{n_{+}}{\mu})<0.
\end{split}
\right. 
\ee
The eigenvectors of $\l_{1},\l_{2},\l_{3}$ are obtained respectively as follows
\be
r_{1}=
\begin{pmatrix}
1\\ 
\l_{1}\\ 
-\f{\mu}{n_{+}}(\l_{1}^{2}-\f{\r_{+}u^{2}_{+}-A_{1}\gm\r^{\gm}_{+} }{\mu u_{+}}\l_{1}-\f{n_{+}}{\mu})
\end{pmatrix},
r_{2}=
\begin{pmatrix}
	1\\ 
	\l_{2}\\ 
	-\f{\mu}{n_{+}}(\l_{2}^{2}-\f{\r_{+}u^{2}_{+}-A_{1}\gm\r^{\gm}_{+} }{\mu u_{+}}\l_{2}-\f{n_{+}}{\mu})
\end{pmatrix},
r_{3}=
\begin{pmatrix}
	1\\ 
0\\ 
1
\end{pmatrix}
\label{ev}
\ee
Hence, we have
\be
\boldsymbol{P}=[r_{1},r_{2},r_{3}].
\label{em}
\ee
With the help of  manifold theorem \cite{C}, there exist a local center manifold $W^{c}(0,0,0)$ and a local stable manifold $W_{3}^{s}(0,0,0)$ 
\begin{align}
&
W^{c}(0,0,0)=\{(z_{1},z_{2},z_{3})~ |~ z_{1}=f^{c}_{1}(z_{3}), z_{2}=f^{c}_{2}(z_{3}),  |z_{3}|~{\rm sufficiently~small}   \},
\label{rd}
\\
&
W_{3}^{s}(0,0,0)=\{(z_{1},z_{2},z_{3})~ |~ z_{1}=f^{s}_{1}(z_{2}), z_{3}=f^{s}_{2}(z_{2}),  |z_{2}|~{\rm sufficiently ~small}   \},
\end{align}
where $f^{c}_{i}, f^{s}_{i}, i=1,2$ are smooth functions and $f^{c}_{i}(0)=0,~ Df^{c}_{i}(0)=0,~f^{s}_{i}(0)=0,~ Df^{s}_{i}(0)=0,~i=1,2$.
With  $(\b{u},\b{w},\b{v})^{T}=\boldsymbol{P}(z_{1},z_{2},z_{3})^{T}$, $(\ref{g})$,  and $(\ref{bg})$, we   gain
 \be
 \b{g}_{3}(z^{3})=az^{2}_{3}+O(|z_{1}|^{2}+|z_{2}|^{2}+|z_{3}|^{3}+|z_{1}z_{3}|+|z_{2}z_{3}|),
 \label{g3}
 \ee
 where 
 \be 
 a=\f{A_{1}\gm(\gm+1) \r^{\gm}_{+}+A_{2}\a (\a+1)n^{\a}_{+}}{2u^{2}_{+}(\mu+n_{+})(1+b^{2})}>0,\quad b=\f{\r_{+}(u_{+}^{2}-A_{1}\gm\r_{+}^{\gm-1})}{ |u_{+}|\sqrt{(\mu+n_{+})n_{+}}}
\ee 
 The system $(\ref{sf z})$  can be reformulated as follows
 \be
 \left\lbrace 
 \begin{split}
&z_{1x}=\l_{1}z_{1}+\b{g}(z_{1},z_{2},z_{3}),\\&
 	z_{2x}=\l_{2}z_{2}+\b{g}(z_{1},z_{2},z_{3}),\\&
 	z_{3x}=az_{3}^{2}+O(|z_{1}|^{2}+|z_{2}|^{2}+|z_{3}|^{3}+|z_{1}z_{3}|+|z_{2}z_{3}|).
 \label{z}
 \end{split}
\right. 
 \ee
 Let $\sigma_{1}(x)$ be a solution to $({\ref{z}})_{1}$ restricted on the local center manifold satisfying the equation 
 \be 
 	\sigma_{1x}=a\sigma_{1}^{2}+O(\sigma_{1}^{3}),\quad \sigma_{1}(x) \to 0 ~{\rm as}~ x \to +\infty.
 	\label{sg}
 \ee 
 which implies  that  there exists the  monotonically increasing solution $\sigma_{1}(x)<0$ to $(\ref{sg})$   for  $\sigma_{1}(0)<0$ and $|\sigma_{1}(0)|$ sufficiently small.  
 Therefore, if the initial data $(z_{1-},z_{2-},z_{3-})$ belongs to the region $\mathcal{M}\subset \mb{R}^{3}$ associated to the local stable manifold and the local center manifold, then  we have 
 \be 
 \begin{aligned}
 &
 z_{i}=O(\sigma_{1}^{2})+O(\delta e^{-cx}),~i=1,2
 \\
 &
 z_{3}=\sigma_{1}+O(\delta e^{-cx}),
 \end{aligned}
 \ee 
with $z_{3-}<0$, the smallness of $|(z_{1-},z_{2-},z_{3-})|$ and 
 \be
 c\f{\delta}{1+\delta x}\leq |\sigma_{1}| \leq C\f{\delta}{1+\delta x},\quad   |\partial ^{k}\sigma_{1}|\leq C \f{\delta^{k+1}}{(1+\delta x)^{k+1}}, \quad C>0,~~k=0,1,2,3...
 \ee  
 It is easy to get
 \be 
 |\partial_{x}^{k}(\wt{\r}-\r_{+},\wt{u}-u_{+}, \wt{n}-n_{+}, \wt{v}-u_{+})|\leq C  \f{\delta^{k+1}}{(1+\delta_{x}x)^{k+1}},\quad C>0,~~k=0,1,2,3...
 \ee 
 \be 
 (\wt{u}-u_{+}, \wt{v}-u_{+})_{x}=(a, a )\sigma^{2}_{1}+O(|\sigma_{1}|^{3})
 \ee
 with the help of $(\ref{trans})$, $(\ref{ev})$ and $(\ref{em})$. 
%


\section{Asymptotic stability of stationary solutions}
	 
\label{sec:$M_{+} >1$}
%
The  function space $Y(0,T)$ for $T>0$ is denoted by
\begin{equation}
\begin{aligned}
Y(0,T)=\{
~(\varphi, \eta, \psi)~ |~&
( \varphi, \eta,  \psi) \in C([0,T]; H^{1}(\mb{R}_{+})), \\
&( \varphi_{x},   \eta_{x}) \in L^{2}([0,T]; L^{2}(\mb{R}_{+})), ~\psi_{x} \in L^{2}([0,T]; H^{1}(\mb{R}_{+}))~
\}.
\end{aligned}
\end{equation}
Let
\be
\phi=\r-\wt{\r},\quad \psi=u-\wt{u},\quad \b{\h}=n-\wt{n},\quad \b{\s}=v-\wt{v}.
\ee 
Then the pertubation $(\phi,\psi,\b{\h},\b{\s})$ satisfies the following system
\be 
\left\lbrace 
\begin{aligned}
&	
\h_{t}+u\h_{x}+\r\s_{x}=-(\s\wt{\r}_{x}+\h\wt{u}_{x}),
\\
& 
\s_{t}+u\s_{x}+\f{p{'}_{1}(\r)}{\r}\h_{x}-\f{\mu\s_{xx}}{\r}-\f{n(\b{\s}-\s)}{\r}
=F_{1},
\\
&
\b{\h}_{t}+v\b{\h}_{x}+n\b{\s}_{x}
=-(\b{\s}\wt{n}_{x}+\b{\h}\wt{v}_{x}),\\&
\b{\s}_{t}+v\b{\s}_{x}+\f{p{'}_{2}(n)}{n}\b{\h}_{x}-\f{(n\b{\s}_{x})_{x}}{n}+(\b{\s}-\s)
=F_{2},
\label{f1}
\end{aligned}
\right. 
\ee 
where 
\be 
F_{1}=-[-\mu(\f{1}{\r}-\f{1}{\wt{\r}})\wt{u}_{xx} +\s\wt{u}_{x} +(\f{p{'}_{1}(\r)}{\r}-\f{p{'}_{1}(\wt{\r})}{\wt{\r}}) \wt{\r}_{x}-(\f{n}{\r}-\f{\wt{n}}{\wt{\r}})(\wt{v}-\wt{u})],
\label{F1}
\ee 
\be 
F_{2}=-[-(\f{1}{n}-\f{1}{\wt{n}})(\wt{n}\b{\s}_{x})_{x}-\f{(\b{\h}\wt{v}_{x})_{x}}{n}+\b{\s}\wt{v}_{x}+(\f{p{'}_{2}(n)}{n}-\f{p{'}_{2}(\wt{n})}{\wt{n}}) \wt{n}_{x}].~~~~~~~
\label{F2}
\ee 
The initial and boundary conditions to the system $(\ref{f1})$ satisfy
\be 
(\phi,\psi,\b{\h},\b{\s})(0,x):=(\h_{0},\s_{0},\b{\h}_{0},\b{\s}_{0})=(\r_{0}-\wt{\r},u_{0}-\wt{u},n_{0}-\wt{n},v_{0}-\wt{v}),
\label{intial d1}
\ee 
\be 
\lim_{x\to \infty}(\h_{0},\s_{0},\b{\h}_{0},\b{\s}_{0})=(0,0,0,0),\quad (\s,\b{\s})(t,0)=(0,0).
\label{boundary d1}
\ee 
\begin{prop}
	\label{prop time decay}
	Assume that the same assumptions  in Theorem $\ref{thm long time behavior}$ hold.
 Let $(\h,\s,\b{\h},\b{\s})$  be the solution to the problem $(\ref{f1})$-$(\ref{boundary d1})$ satisfying 
	$ (\h,\s,\b{\h},\b{\s}) \in Y(0,T)  $  for a certain positive constant T. Then there exist positive constants $\varepsilon$ and C independent of T such that if 
	\be \sup_{0\leq t\leq T}\|  (\h,\s,\b{\h},\b{\s}) \|_{1}+\delta\leq \v \label{priori e}\ee
	is satisfied, it holds for an arbitrary $t\in [0,T]$ that 	
	\begin{equation}
	\begin{split}
	&\| (\h,\s,\b{\h},\b{\s}) \|_{1}+\int_{0}^{t} \|  (\h_{x},\s_{x},\b{\h}_{x},\b{\s}_{x}) \|^{2}d\tau+\int_{0}^{t}\| (\b{\s}-\s,\s_{xx},\b{\s}_{xx})  \|^{2}d\tau \leq C\|(\h_{0},\s_{0},\b{\h}_{0},\b{\s}_{0}) \|_{1}^{2}.
	\end{split}
	\end{equation}
\end{prop}  
With $(\ref{priori e})$,
it is easy to verify  the following Sobolev inequality 
\begin{equation}
~~~~~|h(x) |   \leq   \sqrt{2} \| h  \|^{\frac{1}{2}} \| h_{x} \|^{\frac{1}{2} }       ~~~~{\rm for}~~  h(x) \in H^{1}(\mathbb{R}_{+}).
\end{equation} 
\begin{lem}[ \cite{KNNZ out} ]
	\label{lem d2}
	For any function $\s(\cdot,t)\in H^{1}(\mb{R}_{+})$,  it holds
	\be   
	\delta	\int_{0}^{\infty} e^{-c_{0}x}|\s|^{2}dx\leq C\delta(|\s(0,t)|^{2}+\| \s_{x}(t) \|^{2}),	
	\label{d'1}
	\ee
	\be
	\int_{0}^{\infty}\f{\delta^{j}}{(1+\delta x)^{j}}|\s|^{2}dx\leq C\delta^{j-2}(|\s(0,t)|^{2}+\| \s_{x}(t) \|^{2}), \quad {\rm for}~ j>2,
	\label{d'2}
	\ee  
	where $\delta>0$, $c_{0}>0$, $C>0$ are positive constants.
\end{lem}

With the above Lemma, we can gain the $L^{2}$ estimates of $(\h,\s,\b{\h},\b{\s})$.
\begin{lem}
	\label{lem e0}
	Under the same conditions in Proposition $\ref{prop time decay}$, then the solution $(\h,\s,\b{\h},\b{\s})$ to the problem $(\ref{f1})$-$(\ref{boundary d1})$ satisfies  for $t \in[0,T]$
	\begin{equation}
	\begin{split}
	&
	\| (\h,\s,\b{\h},\b{\s}) \|^{2}+\int_{0}^{t}\|(\s_{x},\b{\s}_{x},\b{\s}-\s)  \|^{2}d\tau\leq C \| (\h_{0},\s_{0},\b{\h}_{0},\b{\s}_{0}) \|^{2}+C(\delta+\v)\int_{0}^{t}\| (\h_{x},\b{\h}_{x}) \|^{2}d\tau.
	\label{e0}
	\end{split}
	\end{equation}
\end{lem}
\begin{proof}
Define 
\be 
\begin{aligned}
&
 \Phi_{1}=\r\int_{\wt{\r}}^{\r}\f{p_{1}(s)-p_{1}(\wt{\r})}{s^{2}}ds,\quad 
     \mathcal{E}_{1}=\r(\f{\s^{2}}{2}+\Phi_{1}),
     \\
     &
 \Phi_{2}=n\int_{\wt{n}}^{n}\f{p_{2}(s)-p_{2}(\wt{n})}{s^{2}}ds,\quad
   \mathcal{E}_{2}=n(\f{\b{\s}^{2}}{2}+\Phi_{2}). 
 \end{aligned}
\nonumber
  \ee
Then, by $(\ref{f})$ and $(\ref{stationary f})$, we gain 
\be 
\begin{split}
&\quad (\mc{E}_{1}+\mc{E}_{2})_{t}+G_{x}+n(\s-\b{\s})^{2}+\mu \s_{x}^{2}+n\b{\s}^{2}_{x}+n(\b{\s}-\s)^{2}
+R_{1}+R_{2}=-R_{3},
\end{split}
\label{f_0}
\ee 
where
\be 
\begin{aligned}
&G:=[u\mc{E}_{1}+v\mc{E}_{2}+(p{'}_{1}(\r)-p{'}_{1}(\wt{\r}))\s+(p{'}_{2}(n)-p{'}_{2}(\wt{n}))\b{\s}]_{x}-[\mu \s\s_{x}+n\b{\s}\b{\s}_{x}+\b{\h}\b{\s}\wt{v}_{x}],
\\
&
R_{1}:= [\r\s^{2}+p_{1}(\r)-p_{1}(\wt{\r})-p{'}_{1}(\wt{\r}) \h]\wt{u}_{x} +[n\b{\s}^{2}+p_{2}(n)-p_{2}(\wt{n})-p{'}_{2}(\wt{n}) \b{\h}]\wt{v}_{x},
\\
& 
 R_{2}:=\h\s\f{\wt{\r}\wt{u}\wt{u}_{x}+(p_{1}(\wt{\r}))_{x}}{\wt{\r}}+\b{\h}\b{\s}\f{\wt{n}\wt{v}\wt{v}_{x}+(p_{2}(\wt{n}))_{x}}{\wt{n}},
 \\
 &
 R_{3}:= \b{\h}(\b{\s}-\s)(\wt{v}-\wt{u})+\b{\h}\b{\s}_{x}\wt{v}_{x}.
 \nonumber
\end{aligned}
\ee 
Integrating $(\ref{f_0})$ over $(0,\infty)$, we get 
\be 
\f{d}{dt}\int \mathcal{E}_{1}+\mathcal{E}_{2}dx+\int n(\s-\b{\s})^{2}+\mu \s_{x}^{2}+n\b{\s}^{2}_{x}+n(\b{\s}-\s)^{2}dx
+\int R_{1}dx+\int R_{2}dx=-\int R_{3}dx,
\label{f_}
\ee 
where we have used $(\ref{boundary d1}).$
With the help of  $(\ref{M_{+}>1 stationary solution d})$, $(\ref{priori e})$ and Sobolev embedding  inequality, we obtain
	\be 
	\int n(\s-\b{\s})^{2}+\mu \s_{x}^{2}+n\b{\s}\b{\s}^{2}dx \geq C \| (\s-\b{\s},\s,\b{\s}) \|^{2}-C(\v+\delta )\| (\s-\b{\s},\b{\s}) \|^{2},
	\label{l_2}
	\ee 
For the  case $M_{+}\neq 1$ in Theorem \ref{thm long time behavior}, with  $(\ref{M_{+}>1 stationary solution d})$  and Sobolev embedding  inequality, we have
\be 
\int_{0}^{\infty} |R_{1}+R_{2}+R_{3}| dx \leq C\delta \|(\h_{x},\b{\h}_{x},\s_{x},\b{\s}_{x},\b{\s}-\s) \|^{2}.
\label{case 1}
\ee 
 For the case $M_{+}=1$ in Theorem \ref{thm long time behavior}, 
with  $(\ref{sigma})$ and Sobolev embedding  inequality, we obtain
\be 
\begin{aligned}
\int R_{1}+R_{2} +R_{3}dx \geq&
\int (\r_{+}\s^{2}+\f{p{''}_{1}(\r_{+})}{2}\h^{2}+\f{u_{+}^{2}-p{'}_{1}(\r_{+})}{u_{+}}\h\s)\wt{u}_{x}+(n_{+}\b{\s}^{2}+\f{p{''}_{2}(n_{+})}{2}\b{\h}^{2}
\\
&
 +\f{u_{+}^{2}-p^{'}_{2}(n_{+})}{u_{+}}\b{\h}\b{\s}) \wt{v}_{x} dx-C(\v+\delta)\|(\h_{x},\s_{x},\b{\h}_{x},\b{\s}_{x},\b{\s}-\s) \|^{2}
\\
\geq&
 -C(\v+\delta)\|(\h_{x},\s_{x},\b{\h}_{x},\b{\s}_{x},\b{\s}-\s) \|^{2},
	\label{case 2}
\end{aligned}
\ee 
where we use 
 $|p{'}_{1}(\r_{+})-p{'}_{2}(n_{+})|<
\sqrt{2} |u_{+}|\min\{(1+\f{\r_{+}}{n_{+}})[ (\gm-1))p{'}_{1}(\r_{+}) ]^{\f{1}{2}},(1+\f{n_{+}}{\r_{+}})[(\a-1)p{'}_{2}(n_{+})]^{\f{1}{2}} \}$ and take $\delta$ and $\varepsilon$ small enough.\\	
Integrating  $(\ref{f_0})$ over $\mb{R}_{+}\times[0,t]$, substituting $(\ref{case 1})$ or $(\ref{case 2})$ into the resulted equation, we have 
\be 
\begin{aligned}	
	&
\| (\h,\s,\b{\h},\b{\s}) \|^{2}+\int_{0}^{t}\|(\s_{x},\b{\s}_{x},\b{\s}-\s)  \|^{2}d\tau
\\
\leq&
 C \| (\h_{0},\s_{0},\b{\h}_{0},\b{\s}_{0}) \|^{2}+C(\delta+\v)\int_{0}^{t}\| (\h_{x},\b{\h}_{x}) \|^{2}d\tau,
	\end{aligned}\ee  
where we take $\v$ and $\delta$ sufficiently small.
%
Hence, we  complete the proof of Lemma \ref{lem e0}.
\end{proof}

In order to complete the proof of Proposition $\ref{prop time decay}$, we need to obtain  estimates of  $(\h_{x}, \s_{x}, \b{\h}_{x},\b{\s}_{x})$.
\begin{lem}
	\label{lem e1}
Under the same conditions in Proposition $\ref{prop time decay}$, then the solution $(\h,\s,\b{\h},\b{\s})$ to the problem $(\ref{f1})$-$(\ref{boundary d1})$ satisfies for $t \in[0,T]$
		\begin{equation}
	\begin{split}
&\| (\h_{x},\b{\h}_{x})  \|^{2}+\int_{0}^{t} \| (\h_{x},\b{\h}_{x}) \|^{2}	\leq C\| (\h_{0},\s_{0},\b{\h}_{0},\b{\s}_{0},\h_{0x},\b{\h}_{0x}) \|^{2}+C\int_{0}^{t}(\v+\delta)\|  (\s_{xx},\b{\s}_{xx})\|^{2}.
\label{1-order time e1}
\end{split}
\end{equation}
\end{lem}
\begin{proof}
Differentiating $(\ref{f1})_{1}$ in $x$, then multiplying the resulted equation by $\mu \h_{x}$, $(\ref{f1})_{2}$ by $\wt{\r}^{2}\h_{x}$ respectively, we gain
\begin{align}
& (\mu \f{\h_{x}^{2}}{2})_{t}+(\mu u\f{\h_{x}^{2}}{2})_{x}+\mu \wt{\r}\h_{x}\s_{xx}
=-\mu [\f{3}{2}\s_{x}\h_{x}^{2}+(\h\s_{xx}+\f{1}{2}\h_{x}\wt{u}_{x}+\s_{x}\wt{\r}_{x})\h_{x}+(\h\wt{u}_{x}+\s\wt{\r}_{x})_{x}\h_{x}],
\label{h_{x}}
\\
& (\wt{\r}^{2}\h_{x}\s)_{t}-(\wt{\r}^{2}\h_{t}\s)_{x}+2\wt{\r}\wt{\r}_{x}\h_{t}\s+\wt{\r}^{2}\h_{t}\s_{x}+\wt{\r}^{2}u\h_{x}\s_{x}+\wt{\r}^{2}\f{p^{'}_{1}(\r)}{\r}\h_{x}^{2}-\mu \wt{\r}\h_{x}\s_{xx}-\wt{\r}^{2}\f{n}{\r}(\b{\s}-\s)\h_{x}
\notag
\\
=&\mu\wt{\r}^{2}(\f{1}{\r}-\f{1}{\wt{\r}})\h_{x}\s_{xx}+F_{1}\wt{\r}^{2}\h_{x}.
\label{s_{x}}
\end{align}

Similarly, differentiating $(\ref{f1})_{3}$ in $x$, then multiplying the resultant equation by $\b{\h}_{x}$, $(\ref{f1})_{4}$ by $\wt{n}\b{\h}_{x}$ respectively lead to
\begin{align} &(\f{\b{\h}_{x}^{2}}{2})_{t}+(v\f{\b{\h}_{x}^{2}}{2})_{x}+\wt{n}\b{\h}_{x}\b{\s}_{xx}=-[\f{3}{2}\b{\s}_{x}\b{\h}_{x}^{2}+(\b{\h}\b{\s}_{xx}+\f{1}{2}\b{\h}_{x}\wt{v}_{x}+\b{\s}_{x}\wt{n}_{x})\b{\h}_{x}-(\b{\h}\wt{v}_{x}+\b{\s}\wt{n}_{x})_{x}\b{\h}_{x}],
\\
&
(\wt{n}\b{\h}_{x}\b{\s})_{t}-(\wt{n}\b{\h}_{t}\b{\s})_{x}+\wt{n}_{x}\b{\h}_{t}\b{\s}+\wt{n}\b{\h}_{t}\b{\s}_{x}+\wt{n}u\b{\h}_{x}\b{\s}_{x}+\wt{n}\f{p^{'}_{2}(n)}{n}\b{\h}_{x}^{2}-\wt{n}\b{\h}_{x}\b{\s}_{xx}+\wt{n}(\b{\s}-\s)\b{\h}_{x}
	\notag
\\
=&
\wt{n}\f{(\b{\h}\b{\s}_{x})_{x}}{n}\b{\h}_{x}+\wt{n}(\f{1}{n}-\f{1}{\wt{n}})(\wt{n}\b{\s}_{x})_{x}+\wt{n}_{x}\b{\h}_{x}\b{\s}_{xx}+F_{2}\wt{n}\b{\h}_{x}.
	\label{bs_{x}}
\end{align} 
 Adding $(\ref{h_{x}})$-$(\ref{bs_{x}})$ together  and integrating the resulted equation in $x$ over $[0,\infty)$ lead to
\be 
\begin{split} 
&\f{d}{dt}\int (\mu \f{\h_{x}^{2}}{2}+\f{\b{\h}_{x}^{2}}{2}+\wt{\r}^{2}\h_{t}\s +\wt{n}\b{\h}_{t}\b{\s})dx+\int(\mu u\f{\h_{x}^{2}}{2}+v\f{\b{\h}_{x}^{2}}{2}-\wt{\r}^{2}\h_{t}\s-\wt{n}\b{\h}_{t}\b{\s})_{x}dx\\&+\int (\wt{\r}^{2}\f{p^{'}_{1}(\r)}{\r}\h_{x}^{2}+\wt{n}\f{p^{'}_{2}(n)}{n}\b{\h}_{x}^{2})dx
=\sum_{i=1}^{4}I_{i},
\label{1f}
\end{split}
\ee 
where
\be 
\begin{split}
	&
I_{1}=-\int[ \wt{\r}^{2}(\h_{t}+u\h_{x})\s_{x}+\wt{n}(\b{\h}_{t}+v\b{\h}_{x})+2\wt{\r}\wt{\r}_{x}\h_{t}\s+\wt{n}_{x}\b{\h}_{t}\b{\s}]dx,\\&
I_{2}=\int -\mu \h\h_{x}\s_{xx}-\b{\h}\b{\h}_{x}\b{\s}_{xx}+ \mu\wt{\r}^{2}(\f{1}{\r}-\f{1}{\wt{\r}})\h_{x}\s_{xx}+\wt{n}(\f{1}{n}-\f{1}{\wt{n}})(\wt{n}\b{\s}_{x})_{x} +\wt{\r}^{2}\f{n}{\r}(\b{\s}-\s)\h_{x}-\wt{n}(\b{\s}-\s)\b{\h}_{x}dx,
\\&
I_{3}=-\int \mu \f{3}{2}\s_{x}\f{\h_{x}^{2}}{2}+\f{3}{2}\b{\s}_{x}\b{\h}_{x}^{2}dx+\int \wt{n}\f{(\b{\h}\b{\s}_{x})_{x}}{n}\b{\h}_{x}dx,\quad I_{4}=\int F_{1}\wt{\r}^{2}\h_{x}+F_{2}\wt{n}\b{\h}_{x}dx,\\&
I_{5}=-\int [\mu (
\f{1}{2}\h_{x}\wt{u}_{x}+\s_{x}\wt{\r}_{x})\h_{x}+\mu (\h\wt{u}_{x}+\s\wt{\r}_{x})_{x}\h_{x}
\\&\quad\quad
 -\wt{n}_{x}\b{\h}_{x}\b{\s}_{xx}+(
\f{1}{2}\b{\h}_{x}\wt{v}_{x}+\b{\s}_{x}\wt{n}_{x})\b{\h}_{x}-(\b{\h}\wt{v}_{x}+\b{\s}\wt{n}_{x})_{x}\b{\h}_{x}]dx.
\nonumber
\end{split}
\ee 
We estimate terms in the left side of $(\ref{1f})$. The second term in the left side is estimated as follows
\begin{equation}
\begin{split}
&\int (\mu u\f{\h_{x}^{2}}{2}+v\f{\b{\h}_{x}^{2}}{2}-\wt{\r}^{2}\h_{t}\s-\wt{n}\b{\h}_{t}\b{\s})_{x}dx=- u_{-}\f{\mu \h_{x}^{2}(0,t)+\b{\h}_{x}^{2}(0,t)}{2}
\\&~~~~~~~~~~~~~~~~~~~~~~~~~~~~~~~~~~~~~~~~~~~~~~~\quad  \leq C\|(\s_{x},\b{\s}_{x}) \|^{2}+\e \|(\s_{xx},\b{\s}_{xx})  \|^{2},
 \end{split}
 \label{h(0)}
 \end{equation}
 where we have used $u_{-}\h_{x}(0,t)+\r_{-}\s_{x}(0,t)=0$, $u_{-}\b{\h}_{x}(0,t)+n_{-}\b{\s}_{x}(0,t)=0$ and $(\ref{boundary d1})$.
 The third term is estimated as follows
 \be 
 \begin{split}
&\quad  \int \wt{\r}^{2}\f{p{'}_{1}(\r)}{\r}\h_{x}^{2}+\wt{n}\f{p{'}_{2}(n)}{n}\b{\h}_{x}^{2}dx\\&
  \geq \r_{+}p{'}_{1}(\r_{+})\| \h_{x} \|^{2}+p{'}_{2}(n_{+})\| \b{\h}_{x} \|^{2}-C(\| (\h,\b{\h}) \|_{L^{\infty}}+\delta)\| (\h_{x},\b{\h}_{x}) \|^{2}\\&
  \geq\r_{+}p{'}_{1}(\r_{+})\| \h_{x} \|^{2}+p{'}_{2}(n_{+})\| \b{\h}_{x} \|^{2}-C(\v+\delta)\| (\h_{x},\b{\h}_{x}) \|^{2}.
  \end{split}
 \ee 
Then, we turn to estimate terms in the right hand side of $(\ref{1f})$. With Cauchy-Schwartz inequality, Young inequality, $(\ref{f1})_{1}$, $(\ref{f1})_{3}$, we obtain 
 \begin{align}
 |I_{1}|
\leq &C \|(\s_{x},\b{\s}_{x})\|^{2}+C\delta\|(\h_{x},\b{\h}_{x})\|^{2},
\\
 |I_{2}|\leq &
 C\|(\h,\b{\h})\|_{L^{\infty}}\| (\h_{x},\b{\h}_{x},\s_{xx},\b{\s}_{xx}) \|^{2}+C_{\e}\|\b{\s}-\s \|^{2}+\eta\| (\h_{x},\b{\h}_{x}) \|^{2}+C\delta\| (\b{\h}_{x},\b{\s}_{x}) \|^{2}
 \notag
 \\
 \leq & 
 C(\v+\delta+\e)\| (\h_{x},\b{\h}_{x}) \|^{2}+C\v\| (\s_{xx},\b{\s}_{xx}) \|^{2}+C_{\e}\| \b{\s}-\s \|^{2}+C\delta\|\b{\s}_{x}  \|^{2},
\\
 |I_{3}|\leq &
 C\| (\s_{x},\b{\s}_{x}) \|_{L^{\infty}} \| (\h_{x},\b{\h}_{x}) \|^{2}+C\| \b{\h} \|_{L^{\infty}}\| (\b{\h}_{x},\b{\s}_{xx}) \|^{2}
 \notag
 \\
 \leq&
  C(\|(\s_{x},\b{\s}) \|+\|(\s_{xx},\b{\s}_{xx}) \|)\| (\h_{x},\b{\h}_{x}) \|^{2}+\v\| (\b{\h}_{x},\b{\s}_{xx}) \|^{2}
  \notag
  \\
 \leq &
 C\v \| (\s_{x},\b{\s}_{x},\h_{x},\b{\h}_{x}) \|^{2}+C\v\|( \s_{xx},\b{\s}_{xx}) \|^{2},
\\
|I_{4}+&I_{5}|\leq 
C\delta\| (\h_{x},\s_{x},\b{\h}_{x},\b{\s}_{x}) \|^{2} .
\label{I_{5}}
\end{align} 
Finally, the substitution of $(\ref{h(0)})$-$(\ref{I_{5}})$ into $(\ref{1f})$ leads to
\begin{equation}
\begin{split}
&
\f{d}{dt}\int(\h_{x}^{2}+\b{\h}_{x}^{2}+\wt{\r}^{2}\h_{x}\s +\wt{n}\b{\h}_{x}\b{\s})dx+\| (\h_{x},\b{\h}_{x})\|^{2}\leq C\| (\s_{x},\b{\s}_{x},\b{\s}-\s_{x}) \|^{2}+C\v\| (\s_{xx},\b{\s}_{xx}) \|^{2}
\label{e1}
\end{split}
\end{equation}
where we let $\delta$, $\v$ and $\e$ small enough.
 Integrating $(\ref{e1})$  
  over $[0,t]$, we obtain
\begin{equation}
\begin{split}
&\|(\h_{x},\b{\h}_{x}) \|^{2}+\int_{0}^{t}\| (\h_{x},\b{\h}_{x}) \|^{2}d\tau\leq C\|  (\h_{0},\b{\h}_{0},\s_{0},\b{\s}_{0},\h_{0x},\b{\h}_{0x}) \|^{2}+C\v\int_{0}^{t}\| (\s_{xx},\b{\s}_{xx}) \|^{2}
\label{high order time e1}	,
\end{split}
\end{equation}
where we have used $(\ref{e0})$ and  Cauchy-Schwarz inequality. Hence, we complete the proof of $(\ref{lem e1})$.
\end{proof}

\begin{lem}
\label{lem psi_{xx} e1}
Assume  that the same conditions in Proposition $\ref{prop time decay}$ hold, then the solution $(\h,\s,\b{\h},\b{\s})$ to the problem $(\ref{f1})$-$(\ref{boundary d1})$ satisfies  for $t \in[0,T]$
\begin{equation}
\begin{split}
&
~~~~  \|(\psi_{x},\b{\s}_{x} ) \|^{2} 	+\int_{0}^{t}\|( \psi_{xx},\b{\s}_{xx}) \|^{2} 	d\tau
\leq C\|(\h_{0},\s_{0},,\h_{0x},\s_{0x},\b{\h}_{0},\b{\s}_{0},\b{\h}_{0x},\b{\s}_{0x})\|^{2} . 
\end{split}
\end{equation}
\end{lem}
\begin{proof}
Multiplying $(\ref{f1})_{2}$ by $-\psi_{xx}$, $(\ref{f1})_{4}$ by $-\b{\s}_{xx}$ respectively,  then adding them togther and integrating the resulted equation in $x$ over $\mathbb{R}_{+}$ 
to gain
\begin{equation}
\begin{split}
&
\frac{d}{dt} \int \frac{\psi^{2}_{x}}{2}+\f{\b{\s}_{x}^{2}}{2}dx  		+ \int  \mu \frac{1}{\rho}	\psi^{2}_{xx}+\b{\s}_{xx}^{2}  dx  =-\sum^{3}_{i=1} K_{i},
\label{psi_{xx} space w}
\end{split}
\end{equation}
where
\begin{equation}
\begin{aligned}
K_{1}	=&
-\int [ 	u \psi_{x} \psi_{xx} 	+v \b{\s}_{x}\b{\s}_{xx}+\f{n}{\r}(\b{\s}-\s)\s_{xx}-\f{p{'}_{1}(\r)}{\r}\h_{x}\s_{xx}-\f{n}{\r}(\b{\s}-\s)\b{\s}_{xx}	-\f{p_{2}{'}(n)}{n}\b{\h}_{x}\b{\s}_{xx}]	dx,
\\
K_{2} 	=&
\int  [  -\widetilde{u}_{x} \psi \psi_{xx}  	+\mu \widetilde{u}_{xx} ( \frac{1}{\rho}-\f{1}{\wt{\r}})\s_{xx}-(\f{p{'}_{1}(\r)}{\r}-\f{p{'}_{1}(\wt{\r})}{\wt{\r}})\s_{xx}\wt{\r}_{x}+(\f{n}{\r}-\f{\wt{n}}{\wt{\r}})(\wt{v}-\wt{u})\s_{xx} 	
-\widetilde{v}_{x} \b{\s} \b{\s}_{xx}
\\
&
+
 (\wt{n}\wt{v}_{x})_{x} ( \frac{1}{n}-\f{1}{\wt{n}})\b{\s}_{xx}-(\f{p{'}_{2}(n)}{\r}-\f{p{'}_{2}(\wt{n})}{\wt{\r}})\b{\s}_{xx}\wt{n}_{x}+(\f{n}{\r}-\f{\wt{n}}{\wt{\r}})(\wt{v}-\wt{u})\s_{xx}+\f{\wt{n}_{x}}{\wt{n}}\b{\s}_{x}\b{\s}_{xx} +\f{(\b{\h}\wt{v}_{x})_{x}}{n}\b{\s}_{xx}]	dx,
\\
K_{3} =&
-\int  [\f{(\b{\h}\b{\s}_{x})_{x}}{n}\b{\s}_{xx}+(\f{1}{n}-\f{1}{\wt{n}})(\wt{n}\b{\s}_{x})_{x}\b{\s}_{xx}]	dx.
\nonumber
\end{aligned}
\end{equation}
First, we estimate  terms in the left side of $(\ref{psi_{xx} space w})$. 
With $\frac{1}{\rho}=(\frac{1}{\rho}-\frac{1}{\widetilde{\rho}})+(\frac{1}{\widetilde{\rho}}-\frac{1} {  \rho_{+} } )+\frac{1} {  \rho_{+} } $,
the second term  is estimated as follows:
\begin{equation}
\begin{aligned}
\int	\frac{\mu}{\rho} 	\psi_{xx}^{2}+\b{\s}_{xx}^{2}	dx\geq&
 \f{\mu}{\r_{+}}\| \s_{xx} \|^{2}+\| \b{\s}_{xx} \|^{2}-C(\|\h \|_{L^{\infty}}+\delta)\| \s_{xx} \|^{2}
 \\
 \geq&
  \f{\mu}{\r_{+}}\| \s_{xx} \|^{2}+\| \b{\s}_{xx} \|^{2}-C(\v+\delta)\| \s_{xx} \|^{2}.
\label{s_{xx} e}
\end{aligned}
\end{equation}
Then, we turn to estimate each term in the right side of $(\ref{psi_{xx} space w})$. Using $(\ref{M_{+}>1 stationary solution d})$, Sobolev inequality and Cauchy-Schwarz inequality,  we have
\begin{align}
| K_{1}|	\leq 	&
\frac{\mu }{ 16\rho_{+} }	\| \psi_{xx} \|^{2}	+\f{1}{16}\|\b{\s}_{xx} \|^{2}+C \| (\h_{x},\s_{x},\b{\h}_{x},\b{\s}_{x},\b{\s}-\s) \|^{2},
\label{K_{1} e }
\\
|K_{2}|	\leq& 	
C   \delta  \| ( \h_{x},\psi_{x},\b{\h}_{x},\b{\s}_{x})  \|^{2}+C\delta\|(\s_{xx} \b{\s}_{xx} )\|^{2},
\label{K_{2} e }
\\
|K_{3} |	\leq &
C 	\| \b{\h}\|_{L^{\infty}} \| (\b{\s}_{x},\b{\s}_{xx}) \|^{2}+C\| \b{\s}_{x} \|_{L^{\infty}} \| \b{\s}_{xx} \|~\| \b{\h}_{x} \|
\notag
\\
\leq&
 C\v \| (\b{\s}_{x},\b{\s}_{xx}) \|^{2}+C(\| \b{\s}_{x} \|+\|\b{\s}_{xx}\|)\|\b{\s}_{xx}\|~\| \b{\h}_{x}\|
 \notag
 \\
 \leq&
  C\v\| (\b{\s}_{x},\b{\s}_{xx}) \|^{2}.
\label{K_{3} e}
\end{align}
 With $\delta$ and $\varepsilon$ small enough and the substitution of $(\ref{s_{xx} e})$-$(\ref{K_{3} e})$ into $(\ref{psi_{xx} space w})$, we  obtain 
\begin{equation}
\begin{split}
&
~~~~\frac{d}{dt} 	\int	\psi_{x}^{2}+\b{\s}_{x}^{2} dx   	+	\ \frac{\mu}{2\rho_{+}}\|\psi_{xx}\|^{2} +\f{1}{2}	\| \b{\s}_{xx}\|^{2}
\leq 	C \| (\h_{x},\s_{x},\b{\h}_{x},\b{\s}_{x},\b{\s}-\s) \|^{2}	
\label{psi_{xx} e1}
\end{split}
\end{equation}
Integrating  $(\ref{psi_{xx} e1})$  in $\tau$ over $[0,t]$, we gain
\begin{equation}
\begin{split}
&
	\|(\psi_{x},\b{\s}_{x}  \|^{2}	+\int_{0}^{t}\| (\psi_{xx},\b{\s}_{xx}) \|^{2} 	d\tau
\leq   C\|(\h_{0}, \s_{0},\h_{0x}, \s_{0x}, \b{\h}_{0},\b{\s}_{0},\b{\h}_{0x}, \b{
\s}_{0x})\|^{2},
\end{split}
\end{equation}
where we use $(\ref{e0})$, $(\ref{1-order time e1})$ and the smallness of $\varepsilon$.
Then we obtain the desired estimate $(\ref{psi_{xx} e1})$ and complete the proof of Lemma $\ref{lem psi_{xx} e1}$.
\end{proof}

\noindent\textbf{Acknowledgments}

The research of the paper is supported by the National Natural Science Foundation of China (Nos. 11931010, 11871047 and 11671384),  by the key research project of Academy for Multidisciplinary Studies, Capital Normal University, and by the Capacity Building for Sci-Tech Innovation-Fundamental Scientific Research Funds (No. 007/20530290068).


\end{document}